\documentclass[12pt,leqno]{amsart}
\usepackage{amsfonts,amsthm,amsmath,mathtools,amssymb,comment,xcolor,enumitem,tikz}
\usepackage{subcaption}
\usepackage{here}
\usepackage{ascmac}
\usepackage{hyperref}
\usepackage{cleveref}

\usetikzlibrary{graphs,graphs.standard}
\theoremstyle{plain}
\newtheorem{thm}{Theorem}[section]

\newtheorem{lem}{Lemma}[section]
\newtheorem{cor}{Corollary}[section]

\theoremstyle{definition}
\newtheorem{df}{Definition}[section]
\newtheorem{rem}{Remark}[section]

\Crefname{thm}{Theorem}{Theorems}
\Crefname{prop}{Proposition}{Propositions}
\Crefname{lem}{Lemma}{Lemmas}
\Crefname{cor}{Corollary}{Corollaries}
\Crefname{obser}{Observation}{Observations}
\Crefname{spec}{Speculation}{Speculations}
\Crefname{df}{Definition}{Definitions}
\Crefname{rem}{Remark}{Remarks}
\Crefname{ex}{Example}{Examples}
\Crefname{conj}{Conjecture}{Conjectures}
\Crefname{section}{Section}{Sections}
\Crefname{subsection}{Subsection}{Subsections}

\newcommand{\FF}{\mathbb{F}}
\newcommand{\ZZ}{\mathbb{Z}}

\newcommand{\NN}{\mathbb{N}}

\begin{document}

\title{A new approach to pancyclicity of Paley graphs I}

\author{Yusaku Nishimura}
\address{School of Fundamental Science and Engineering, Waseda University, Tokyo 169-8555, Japan}
\email{n2357y@ruri.waseda.jp}
\date{}

\begin{abstract}
  Let $G$ be an undirected graph of order $n$ and let $C_i$ be an $i$-cycle graph. 
  $G$ is called pancyclic if $G$ contains a $C_i$ for any $i\in \{3,4,\ldots,n\}$.
  We show that the pancyclicity of specific Cayley graphs and the Cartesian product of specific two graphs.
  As a corollary of these two theorems, we provide a new proof of the pancyclicity of the Paley graph.
\end{abstract}

\maketitle

{\small
	\noindent
	{\bfseries Key Words:}
	Pancyclic, Generalized Paley graph, Cayley graph%\\ \vspace{-0.15in}
	
	%\noindent
	%2010 {\it Mathematics Subject Classification}. 
	%Primary 11T71;
	%Secondary 94B05, 11F11.\\ \quad
}

\section{Introduction} 

Let $H$ be a finite group, and let $S$ be a subset of $H$ that does not contain the identity element of $H$, is closed under the taking of inverses, 
and generates $H$. 
Then, a graph $G$ can be constructed with the vertex set of $H$, and the vertices $g$ and $h$ are adjacent if $gh^{-1}\in S$.
This graph, $G$, is called a Cayley graph and is denoted as $G=Cay(H,S)$.

Let $\FF_q$ be a finite field with order $q$, and let $(\FF_q^*)^2$, a subset of $\FF_q$, be the set of quadratic residues of $\FF_q$.
It is known that $(\FF_q^*)^2$ generates $\FF_q$ for any $q$, and that $(\FF_q^*)^2$ is closed under the taking of additive inverses if and only if $q\equiv 1\pmod{4}$. 
Therefore, when $q\equiv 1\pmod{4}$, we can consider the Cayley graph $Cay(\FF_q,(\FF_q^*)^2)$. 
The Cayley graph obtained in this case is known as a Paley graph.
Paley graphs are often denoted as $P(q)$, where $q$ is the order of the finite field.
In other words,
\[
  P(q)=Cay(\FF_q,(\FF^*_q)^2).
\]

Paley graphs have many interesting properties. 
For example, B.~Bollobás and A.~Thomason \cite{BOLLOBAS198113} showed that for any simple graph $G$ with order $r$, 
if $q$ is sufficiently large compared to $r$, then $G$ is an induced subgraph of $P(q)$. 
This also implies that there exists $q$ such that $G$ is a subgraph of $P(q)$.
From this property, we can define the Paley index, which is invariant for any simple graphs, as follows.

\begin{df}[\cite{PaleyInd}]\label{def:PaleyInd}
  We say a simple graph $G$ has Paley index $t$ if 
  \[
  t=\min\{q\in \NN\mid \mbox{$G$ is a subgraph of $P(q)$}\}. 
  \]
\end{df}

In this way, Paley graphs are known to possess various properties. 
It is also known that Paley graph $P(q)$ is pancyclic where $q\neq 5$.
Let us take a moment to explain the pancyclic property.

\begin{df}[\cite{BONDY197180}]
  An undirected graph $G$ with order $n\geq 3$ is pancyclic if it contains a $k$-cycle as a subgraph for every $k\in \{3,4,\ldots,n\}$.
\end{df}

Examples of pancyclic graphs include complete graphs and wheel graphs.
J.A.~Bondy put forward the following corollary about the pancyclicity.

\begin{cor}[\cite{BONDY197180}]\label{cor:panc}
  Let $G$ be a graph with order $n$ and edge set $E(G)$.
  If $|E(G)|\geq \frac{n^2}{4}$, then $G$ is either pancyclic or a complete bipartite graph $K_{\frac{n}{2},\frac{n}{2}}$.
\end{cor}

\Cref{cor:panc} implies that we can use counting edges in a graph as a sufficient condition for demonstrating the pancyclicity of any graph.
R.~Matsubara, M.~Tsugaki, and T.~Yamashita proved a stronger theorem.
For any vertex $x$ in a graph, $N(x)$ represents the neighborhood set of $x$ and $d(x)$ represents the degree of $x$.

\begin{thm}[\cite{TYamashita}]\label{thm:tpanc}
  Let $G$ be a 2-connected graph of order $n\geq 6$. 
  Suppose that $|N(x)\cup N(y)|+d(z)\geq n$ for every triple independent vertices $x,y,z$ of $G$.
  Then $G$ is pancyclic or isomorphic to the complete bipartite graph $K_{\frac{n}{2},\frac{n}{2}}$.
\end{thm}

From the properties of the Paley graph, we can easily prove that 
 for any triple independent vertices $x,y,z$ in the Paley graph with order $q\geq 6$, 
 \[|N_G(x)\cup N_G(y)|+d_G(z)=\frac{5q-5}{4}\geq q.\]
Therefore, from \Cref{thm:tpanc} we demonstrate the pancyclicity of the Paley graph.

However, for the generalization of Paley graph, known as generalized Paley graph, \Cref{thm:tpanc} cannot be applied.
\begin{df}[Generalized Paley graph\cite{10.1307/mmj/1242071694}]\label{def:gPaley}
  Let $(\FF_q^*)^k$ be the set of $k$-th powers in $\FF_q^*$, and 
  let $k$ and $q$ be integers that satisfy the conditions that $(\FF_q^*)^k$ is closed under the taking of inverses and generates $\FF_q$.
  The generalized Paley graph is defined as $Cay(\FF_q,(\FF_q^*)^k)$.
\end{df}
In this case, if $k\geq 3$ then for any triple independent vertices $x,y,z$ in the generalized Paley graph, $Cay(\FF_q,(\FF_q^*)^k)$,
\[|N_G(x)\cup N_G(y)|+d_G(z)\leq \frac{q-1}{3}+\frac{q-1}{3}+\frac{q-1}{3}< q.\]
Therefore, we cannot establish the pancyclicity of the generalized Paley graph using \Cref{thm:tpanc}.
The pancyclicity of the generalized Paley graph is an unsolved problem.

We aim to prove the pancyclicity of the generalized Paley graph.
In this paper, we show the pancyclicity of Paley graphs $P(q)$, where $q\neq 5$ using a new method.
Furthermore, we also determine the Paley index of the $n$-cycle from the pancyclicity.
In \cite{next}, we apply this new method to the generalized Paley graph and establish its pancyclicity.

This new method is centered around the following two theorems.
Note that the definitions of several terms in \Cref{thm:carPanc} will be given in \Cref{sec:pre}.

\begin{thm}\label{thm:cayPanc}
  Let $G=Cay(\ZZ/m\ZZ,S)$, where $S$ contains the generator of $\ZZ/m\ZZ$.
  If any pair of vertices in $G$ has at least $1$ common neighbor, 
  and if there exists at least $1$ pair of vertices with at least $2$ common neighbors, then $G$ is pancyclic.
\end{thm}

\begin{thm}\label{thm:carPanc}
  Let $G_1$ be pancyclic and with odd order, and let $G_2$ be a semi-Hamiltonian graph. 
  Then the Cartesian product of these graphs, $G_1\times G_2$, is pancyclic.
\end{thm}

\Cref{thm:cayPanc} implies that counting the common neighbors of any $2$ vertices in a graph can serve as a sufficient condition for demonstrating the pancyclicity of specific graphs.
Since the Paley graph $P(q)$, where $q$ is a prime number not equal to $5$, satisfies the conditions of \Cref{thm:cayPanc}, we have the following corollary.

\begin{cor}\label{cor:paPanc}
  If $q$ is a prime number not equal to $5$, then $P(q)$ is pancyclic.
\end{cor}

On the other hand, \Cref{thm:carPanc} indirectly provides a proof for the pancyclicity of $P(q)$, where $q$ is not a prime number.
Using \Cref{cor:paPanc} and \Cref{thm:carPanc}, we claim to be able to show the pancyclicity of Paley graphs $P(q)$ for $q\neq 5$, and to determine the Paley index of $C_n$. 

\begin{cor}\label{thm:main}
  Any Paley graph without $P(5)$ is pancyclic.
  \end{cor}
\begin{cor}\label{cor:index}
  Let $\rho_{C_n}$ be the Paley index of $C_n$, and let $\lceil n \rceil_{\FF}$ be
  \[
    \lceil n \rceil_{\FF}=\min \{q\in \NN\mid n\leq q\mbox{, where $q$ is a prime power and }q\equiv 1\pmod 4\}.
  \]
  Then, 
  \[
    \rho_{C_n}=
      \begin{cases*}
        \lceil n \rceil_{\FF} & $n\geq 5$\\
        9 & $n<5$
      \end{cases*}.
  \]
\end{cor}
Note that \Cref{thm:cayPanc,thm:carPanc} are not derived by \Cref{cor:panc} or \Cref{thm:tpanc}.

This paper is organized as follows. 
In \Cref{sec:pre}, we give definitions of several terms in \Cref{thm:carPanc} and discuss a property of Paley graphs used in the proof of \Cref{cor:paPanc}.
In \Cref{sec:cay}, we give proofs of \Cref{thm:cayPanc} and \Cref{cor:paPanc}, which concern specific instances of \Cref{thm:main}.
In \Cref{sec:car}, we give a proof of \Cref{thm:carPanc}.
In \Cref{sec:pay}, we give proofs of \Cref{thm:main} and \Cref{cor:index}, using \Cref{cor:paPanc} and \Cref{thm:carPanc}.

\section{Preliminaries}\label{sec:pre}

In the following, $C_n$ denotes an $n$-cycle and $P_n$ denotes an $n$-path. 
$u\sim v$, where $u$ and $v$ are vertices in some graph $G$, signifies that $u$ and $v$ are adjacent in $G$.
If both $X$ and $G$ are graphs, then $X\subset G$ indicates that $X$ is a subgraph of $G$.

\subsection{Definitions of several terms involved in \Cref{thm:carPanc}}

We give the definitions of the Cartesian product of two graphs, the Hamiltonian graph, and the semi-Hamiltonian graph.

\begin{df}
Let $G_1$ and $G_2$ be graphs with vertices $V_1$ and $V_2$, respectively. 
The Cartesian product of $G_1$ and $G_2$ is defined by the vertex set $V_1\times V_2$, 
and the vertices $(u_1,u_2)$ and $(v_1,v_2)$ in $V_1\times V_2$ are adjacent if 
\[
  ((u_1 = v_1) \land (u_2\sim v_2))\lor ((u_1 \sim v_1) \land (u_2 = v_2)).
\]
\end{df}

In the following, we use $G_1\times G_2$ to denote a Cartesian product of two graphs, $G_1$ and $G_2$.

\begin{df}
Let $G$ be an undirected graph with order $v$. 
If $G$ has a subgraph isomorphic to $P_v$, then $G$ is called a semi-Hamiltonian graph, and this subgraph is called a Hamiltonian path.
Similarly, if $G$ has a subgraph isomorphic to $C_v$, then $G$ is called a Hamiltonian graph, and this subgraph is called a Hamiltonian cycle.
\end{df}

Note that from the definitions of a Hamiltonian graph and a pancyclic graph, a graph $G$ is a Hamiltonian graph if it is pancyclic.
In the same way, if a graph $G$ is a Hamiltonian graph, then it is also a semi-Hamiltonian graph.

\subsection{Useful property of Paley graphs}

A Paley graph is a strongly regular graph, a property that we make use of in this paper.
\begin{df}
  A simple graph $G$ is strongly regular if the number of common neighbors between any two vertices depends only on their adjacency.
  A strongly regular graph is characterized by the parameters $(v, k, \lambda, \mu)$, where
  \begin{itemize}
    \item $v$ represents the order of the graph
    \item $k$ represents the degree of each vertex
    \item $\lambda$ represents the number of common neighbors between any two adjacent vertices
    \item $\mu$ represents the number of common neighbors between any two non-adjacent vertices.
  \end{itemize}
\end{df}

A Paley graph $P(q)$ is a strongly regular graph with parameters $(q, \frac{q-1}{2}, \frac{q-5}{4}, \frac{q-1}{4})$~\cite{MR1829620}.

\section{Proofs of \Cref{thm:cayPanc} and \Cref{thm:main} for a specific situation}\label{sec:cay}

In this section, we first give a proof of \Cref{thm:cayPanc},
after which we give one for \Cref{cor:paPanc}, which represents a specific situation of \Cref{thm:main}. 

Note that Paley graph $P(13)$ is an example of a graph that satisfies the conditions of \Cref{thm:cayPanc}.

\begin{proof}[Proof of \Cref{thm:cayPanc}]
  By the definition of $S$, there exists some generator $x$ of $\ZZ/m\ZZ$ such that $x\in S$.
  First, we claim that the assumption $x=1$ can be made without loss of generality.
  If $x\neq 1$, we consider the automorphism $\sigma$ of $\ZZ/m\ZZ$, viewed as an additive group, defined as $\sigma(v)=vx^{-1}$, and let $G'=Cay(\ZZ/m\ZZ,\sigma(S))$.
  Note that $\sigma(S)$ satisfies the conditions of a Cayley graph since $\sigma$ is an automorphism.
  By the definition of a Cayley graph, $u\sim v$ if and only if $u-v\in S$, where $u,v$ is the vertex of $G$, 
  and $u'\sim v'$ if and only if $u'-v'\in \sigma(S)$, where $u',v'\in\sigma(S)$. 
 Because 
  \[
    u-v\in S\Longleftrightarrow\sigma(u-v)=\sigma(u)-\sigma(v)\in\sigma(S),
  \]
  the edge set of $G$ and the edge set of $G'$ are bijective. 
  This implies 
  \[
    Cay(\ZZ/m\ZZ,S)\simeq Cay(\ZZ/m\ZZ,\sigma(S)).
  \] 
  Therefore, if we prove $G'$ is pancyclic, then the theorem holds.
  Since $xx^{-1}=1\in \sigma(S)$, we can assume $x=1$ without loss of generality.
  In the following, we assume $1\in S$.
  
  From the conditions of the theorem, it is easy to see that $C_3,C_4\subset G$.
  Therefore, we focus on the existence of the subgraph $C_n$ in $G$ for all $5 \leq n\leq m$.
  
  We consider the common neighbor of vertices $0$ and $n-2$ in $G$.
  Based on the conditions of the theorem, we know that there is at least one vertex that is a common neighbor of them.
  Let us denote this vertex as $\alpha$. 
  This implies the existence of edges $\{n-2,\alpha\}$ and $\{0,\alpha\}$. 
  Furthermore, based on the given assumption about $S$, there exist edges $\{i,i+1\}$ in $G$, where $i\in \ZZ/m\ZZ$.
  
  If $\alpha > n-2$, we can construct a subgraph $H=(V,E)$ of $G$, where the vertex set $V$ and the edge set $E$ are defined as
\begin{align*}
  V&=\{0,1,2,\ldots,n-2,\alpha \},\\
  E&=\{\{0,1\},\{1,2\},\ldots,\{n-3,n-2\},\{n-2,\alpha\},\{0,\alpha\}\}.
\end{align*}
Obviously, $H$ is isomorphic to $C_n$. 
Therefore, $C_n\subset G$ in this case.
An example with $G=P(13)$, $n-2=5$, and $\alpha=9$ is illustrated in \Cref{figure:case1}.

Next, we consider the case when $1 < \alpha < n-2$.
According to the definition of a Cayley graph, we have $\alpha, \alpha - (n-2) \in S$, and there exist edges 
$\{0, \alpha\}, \{1, \alpha + 1\}, \{2, \alpha + 2\}$, and $\{\alpha + 1, n-1\}$.
Therefore, we can construct a subgraph $H=(V,E)$ of $G$, where the vertex set $V$ and the edge set $E$ are defined as
\begin{align*}
  V&=\{0,1,2,\ldots,n-2,n-1 \},\\
  E&=\left\{
  \begin{aligned}
    &\{0,\alpha \},\{\alpha -1,\alpha\},\{\alpha -2,\alpha -1\}\ldots,\{2,3\},\\
    &\{2,\alpha+2\},\{\alpha+2,\alpha+3\},\{\alpha +3,\alpha+4\},\ldots,\{n-2,n-1\},\\
    &\{\alpha +1,n-1\},\{1,\alpha+1\},\{0,1\}  
  \end{aligned}
  \right\}.
\end{align*}
This choice of edges is constructed in 5 steps:
\begin{itemize}
\item[1] We choose the edge $\{0,\alpha\}$.
\item[2] If $3\leq\alpha$ for all $i$ such that $3\leq i\leq \alpha $, then we choose the edges $\{i-1,i\}$. Otherwise, skip this step.
\item[3] We choose the edge $\{2,\alpha+2\}$.
\item[4] If $\alpha+2\leq n-2$ for all $i$ such that $\alpha+2\leq i\leq n-2$, then we choose the edges $\{i,i+1\}$. Otherwise, skip this step.
\item[5] Finally, we choose the edges $\{\alpha +1,n-1\}$,$\{1,\alpha+1\}$, and $\{0,1\}$.
\end{itemize}
The subgraph constructed in this way can be seen to traverse its vertices exactly once along its edges.
Therefore, $H$ is also isomorphic to $C_n$, and we can conclude that $C_n\subset G$.
An example illustrating this case is shown in \Cref{figure:case2}, where $G=P(13)$, $n-2=6$, and $\alpha=3$.
\begin{figure}[H]
  \begin{minipage}[b]{0.40\columnwidth}
  \centering
  \begin{tikzpicture}
    
    \graph[simple,nodes={draw, circle,minimum size=0.1cm},radius=2cm]{
      subgraph I_n[n=13, clockwise, V={0,1,2,3,4,5,6,7,8,9,10,11,12}];
      {
        [cycle, --]0,1,2,3,4,5,9
      }    
    };

    \foreach \i in {0, 1, 2, 3, 4, 5, 6, 7, 8, 9, 10, 11, 12}
    \node[anchor=center] at (\i) {$\i$};
  \end{tikzpicture}
  \subcaption{$n-2=5,\alpha=9$}\label{figure:case1}
  
\end{minipage}
\hspace{0.05\columnwidth}
\begin{minipage}[b]{0.40\columnwidth}
  \centering
  \begin{tikzpicture}
    
    \graph[simple,nodes={draw, circle,minimum size=0.1cm},radius=2cm]{
      subgraph I_n[n=13, clockwise, V={0,1,2,3,4,5,6,7,8,9,10,11,12}];
      {
        [cycle, --]0,3,2,5,6,7,4,1
      }
    };  

      \foreach \i in {0, 1, 2, 3, 4, 5, 6, 7, 8, 9, 10, 11, 12}
      \node[anchor=center] at (\i) {$\i$};
  \end{tikzpicture}
  
  \subcaption{$n-2=6,\alpha=3$}\label{figure:case2}
  \end{minipage}
  \caption{Example of cycle composition }
\end{figure}
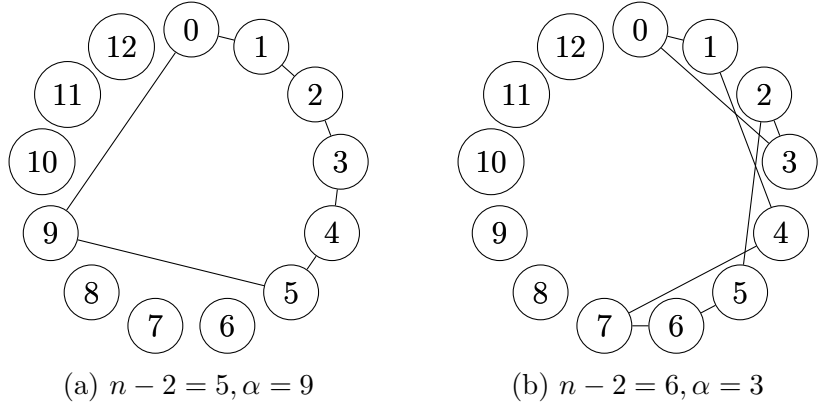

In the case where $\alpha = 1$, we have $-(\alpha - (n-2)) = n-3 \in (\mathbb{F}_q^*)^2$, which implies that $n-3$ is a neighbor of $0$.
Since $1\in S$, $n-3$ is also a neighbor of $n-2$.
Hence, we know $n-3$ is a common neighbor of $0$ and $n-2$.
We are considering the case where $n\geq 5$, so we know that $1 < n-3 < n-2$.
Therefore, if we consider $\alpha$ to be $n-3$, then this situation falls into the case $1 < \alpha < n-2$ that we previously discussed.
We showed in the earlier analysis that $C_n \subset G$ holds for all $\alpha$. 
Hence, the proof is complete.
\end{proof}

Using \Cref{thm:cayPanc}, we prove \Cref{cor:paPanc}.

\begin{proof}[Proof of \Cref{cor:paPanc}]
  A Paley graph $P(q)$ is defined as $Cay(\FF_q,(\FF_q)^2)$, where $q \equiv 1 \pmod{4}$.
  When $q$ is a prime number, we can regard $\FF_q$ as $\ZZ/q\ZZ$ and $(\FF_q)^2$ as a subset of $\ZZ/q\ZZ$.
  Therefore, we can also denote $P(q)$ as $Cay(\ZZ/q\ZZ,(\FF_q)^2)$.
  Since $\pm 1 \in (\FF_q)^2$, $P(q)$ is a Cayley graph that is made from a finite cyclic group and contains a generator.

  The Paley graph is a strongly regular graph, and its parameters are $(q,\frac{q-1}{2},\frac{q-5}{4},\frac{q-1}{4})$.
  The minimum prime number greater than 5 that leaves a remainder of 1 when divided by 4 is 13, so we have $\frac{q-5}{4}\geq 2$.
  This means that any two vertices have at least two common neighbors.

  Hence, we can apply \Cref{thm:cayPanc} to $P(q)$, and the proof is complete.

\end{proof}

\section{pancyclicity of the Cartesian product of specific two graphs}\label{sec:car}

In this section, we prove \Cref{thm:carPanc}.
To do so, we claim the following lemma.

\begin{lem}\label{lem:carPartPanc}
  Let $n$ and $m$ be positive integers.
  Then, $C_{2m+1}\times P_n$ is a Hamiltonian graph. 
  Additionally, for any integer $k$ such that $2m+1\leq k\leq n(2m+1)$, we have $C_k\subset C_{2m+1}\times P_n$.
\end{lem}

\begin{proof}
  By the definition of a Hamiltonian graph, 
  if we prove $C_k\subset C_{2m+1}\times P_n$, where $k$ satisfies the conditions of the lemma, then the proof is complete.
  Let $V_C=\{0,1,\ldots,2m\}$ be the vertices of $C_{2m+1}$, and let $V_P=\{0,1,\ldots,n-1\}$ be the vertices of $P_n$.
By the definition of the Cartesian product, we can regard the vertices of $C_{2m+1}\times P_n$ as $(V_C,V_P)$.
The structure of $C_{2m+1}\times P_n$ is outlined in \Cref{fig:outline}.

  \begin{figure}[H]
    \centering

      \begin{tikzpicture}[scale=0.6]
        \draw[thick] (0,0) grid (2.2,2.2);
        \draw[thick] (0,2.8) grid (2.2,5);
        \draw[thick] (2.8,0) grid (5,2.2);
        \draw[thick] (2.8,2.8) grid (5,5);

        \draw[thick, bend left=30](0,0) to (5,0);
        \draw[thick, bend left=30](0,1) to (5,1);
        \draw[thick, bend left=30](0,2) to (5,2);
        \draw[thick, bend left=30](0,3) to (5,3);
        \draw[thick, bend left=30](0,4) to (5,4);
        \draw[thick, bend left=30](0,5) to (5,5);
      
        \node[scale=0.5] ('0') at (-0.25,-0.25) {$0$};
        \node[scale=0.5] ('1') at (1,-0.5) {$1$};
        \node[scale=0.5] ('2') at (2,-0.5) {$2$};
        \node[scale=0.5] ('2m-2') at (3,-0.5) {};
        \node[scale=0.5] ('2m-1') at (4,-0.5) {$2m-1$};
        \node[scale=0.5] ('2m') at (5,-0.5) {$2m$};
      
        \node[scale=0.5] ('v1') at (-0.5,1) {$1$};
        \node[scale=0.5] ('v2') at (-0.5,2) {$2$};
        \node[scale=0.5] ('vn-3') at (-0.5,3) {};
        \node[scale=0.5] ('vn-2') at (-0.5,4) {$n-2$};
        \node[scale=0.5] ('vn-1') at (-0.5,5) {$n-1$};
      
        \draw[dotted](0,2) to (0,3);
        \draw[dotted](1,2) to (1,3);
        \draw[dotted](2,2) to (2,3);
        \draw[dotted](3,2) to (3,3);
        \draw[dotted](4,2) to (4,3);
        \draw[dotted](5,2) to (5,3);
        \draw[dotted](2,0) to (3,0);
        \draw[dotted](2,1) to (3,1);
        \draw[dotted](2,2) to (3,2);
        \draw[dotted](2,3) to (3,3);
        \draw[dotted](2,4) to (3,4);
        \draw[dotted](2,5) to (3,5);
      \end{tikzpicture}
      \caption{Outline of $C_{2m+1}\times P_n$}\label{fig:outline}
  \end{figure}
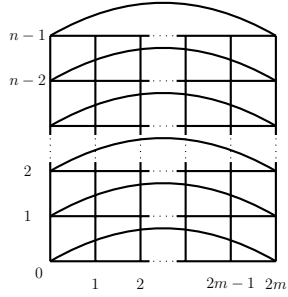
  Let $G=(V_G,E_G)$ be a subgraph of $C_{2m+1}\times P_n$, where $V_G$ is the vertex set and $E_G$ is the edge set.
We will form $C_k$ by carefully choosing vertices and edges from $C_{2m+1}\times P_n$ to construct a subgraph $G$.
First, if $n$ is even, we select $V_G$ and $E_G$ as follows.
  \begin{align*}
    V_G&=(V_P,V_C)\\
    E_G&=\left\{
      \begin{aligned}
        &\{(0,0),(0,1)\},\{(0,1),(0,2)\},\ldots,\{(0,n-2),(0,n-1)\}\\
    &\{(0,n-1),(1,n-1)\},\{(1,n-1),(2,n-1)\},\ldots,\{(2m-1,n-1),(2m,n-1)\}\\
    &\{(2m,n-1),(2m,n-2)\},\{(2m,n-2),(2m-1,n-2)\},\ldots,\{(2,n-2),(1,n-2)\}\\
    &\{(1,n-2),(1,n-3)\},\{(1,n-3),(2,n-3)\},\ldots,\{(2m-1,n-3),(2m,n-3)\}\\
    &\vdots \\
    &\{(1,2),(1,1)\},\{(1,1),(2,1)\},\ldots,\{(2m-1,1),(2m,1)\}\\
    &\{(2m,1),(2m,0)\},\{(2m,0),(2m-1,0)\},\ldots,\{(1,0),(0,0)\}\\
      \end{aligned}
    \right\}.
  \end{align*}
  This choice is represented by the bold lines in \Cref{fig:evenout}.
Similarly, if $n$ is odd, we can choose the edges in the following way.
  \begin{align*}
    V_G&=(V_P,V_C),\\
    E_G&=\left\{
      \begin{aligned}
        &\{(0,0),(0,1)\},\{(0,1),(0,2)\},\ldots,\{(0,n-2),(0,n-1)\}\\
    &\{(0,n-1),(1,n-1)\},\{(1,n-1),(2,n-1)\},\ldots,\{(2m-1,n-1),(2m,n-1)\}\\
    &\{(2m,n-1),(2m,n-2)\},\{(2m,n-2),(2m-1,n-2)\},\ldots,\{(2,n-2),(1,n-2)\}\\
    &\{(1,n-2),(1,n-3)\},\{(1,n-3),(2,n-3)\},\ldots,\{(2m-1,n-3),(2m,n-3)\}\\
    &\vdots \\
    &\{(2m,2),(2m,1)\},\{(2m,1),(2m-1,1)\},\ldots,\{(2,1),(1,1)\}\\
    &\{(1,1),(1,0)\},\{(1,0),(2,0)\},\ldots,\{(2m,0),(0,0)\}\\
      \end{aligned}
    \right\}.
  \end{align*}
  This choice is also indicated in \Cref{fig:oddout}. 
  \begin{figure}[H]
    \begin{minipage}[b]{0.4\textwidth}
      \centering
      \begin{tikzpicture}[scale=0.6]
        \draw[help lines] (0,0) grid (2.2,2.2);
        \draw[help lines] (0,2.8) grid (2.2,5);
        \draw[help lines] (2.8,0) grid (5,2.2);
        \draw[help lines] (2.8,2.8) grid (5,5);
        
        \node[scale=0.5] ('0') at (-0.25,-0.25) {$0$};
        \node[scale=0.5] ('1') at (1,-0.5) {$1$};
        \node[scale=0.5] ('2') at (2,-0.5) {$2$};
        \node[scale=0.5] ('2m-2') at (3,-0.5) {};
        \node[scale=0.5] ('2m-1') at (4,-0.5) {$2m-1$};
        \node[scale=0.5] ('2m') at (5,-0.5) {$2m$};
      
        \node[scale=0.5] ('v1') at (-0.5,1) {$1$};
        \node[scale=0.5] ('v2') at (-0.5,2) {$2$};
        \node[scale=0.5] ('vn-3') at (-0.5,3) {};
        \node[scale=0.5] ('vn-2') at (-0.5,4) {$n-2$};
        \node[scale=0.5] ('vn-1') at (-0.5,5) {$n-1$};

        \draw[very thick](0,0) to (0,2.2);
        \draw[very thick](0,2.8) to (0,5);
        \draw[very thick](0,5) to (2.2,5);
        \draw[very thick](2.8,5) to (5,5);
        \draw[very thick](5,5) to (5,4);
        \draw[very thick](5,4) to (2.8,4);
        \draw[very thick](2.2,4) to (1,4);
        \draw[very thick](1,4) to (1,3);
        \draw[very thick](2.2,3) to (1,3);
        \draw[very thick](2.8,3) to (5,3);
        \draw[very thick](5,2.8) to (5,3);
        \draw[very thick](5,2.2) to (5,2);
        \draw[very thick](2.8,2) to (5,2);
        \draw[very thick](2.2,2) to (1,2);
        \draw[very thick](1,1) to (1,2);
        \draw[very thick](1,1) to (2.2,1);
        \draw[very thick](5,1) to (2.8,1);
        \draw[very thick](5,1) to (5,0);
        \draw[very thick](2.8,0) to (5,0);
        \draw[very thick](2.2,0) to (1,0);
  
        \draw[very thick](0,0) to (1,0);
      
        \draw[dotted](0,2) to (0,3);
        \draw[dotted](1,2) to (1,3);
        \draw[dotted](2,2) to (2,3);
        \draw[dotted](3,2) to (3,3);
        \draw[dotted](4,2) to (4,3);
        \draw[dotted](5,2) to (5,3);
        \draw[dotted](2,0) to (3,0);
        \draw[dotted](2,1) to (3,1);
        \draw[dotted](2,2) to (3,2);
        \draw[dotted](2,3) to (3,3);
        \draw[dotted](2,4) to (3,4);
        \draw[dotted](2,5) to (3,5);
  
      \end{tikzpicture}  
      \subcaption{$n$ even}\label{fig:evenout}
    \end{minipage}
    \hspace{0.1\textwidth}
    \begin{minipage}[b]{0.4\textwidth}
      \centering
      \begin{tikzpicture}[scale=0.6]
        \draw[help lines] (0,0) grid (2.2,2.2);
        \draw[help lines] (0,2.8) grid (2.2,5);
        \draw[help lines] (2.8,0) grid (5,2.2);
        \draw[help lines] (2.8,2.8) grid (5,5);
      
        \node[scale=0.5] ('0') at (-0.25,-0.25) {$0$};
        \node[scale=0.5] ('1') at (1,-0.5) {$1$};
        \node[scale=0.5] ('2') at (2,-0.5) {$2$};
        \node[scale=0.5] ('2m-2') at (3,-0.5) {};
        \node[scale=0.5] ('2m-1') at (4,-0.5) {$2m-1$};
        \node[scale=0.5] ('2m') at (5,-0.5) {$2m$};
      
        \node[scale=0.5] ('v1') at (-0.5,1) {$1$};
        \node[scale=0.5] ('v2') at (-0.5,2) {$2$};
        \node[scale=0.5] ('vn-3') at (-0.5,3) {};
        \node[scale=0.5] ('vn-2') at (-0.5,4) {$n-2$};
        \node[scale=0.5] ('vn-1') at (-0.5,5) {$n-1$};

        \draw[very thick](0,0) to (0,2.2);
        \draw[very thick](0,2.8) to (0,5);
        \draw[very thick](0,5) to (2.2,5);
        \draw[very thick](2.8,5) to (5,5);
        \draw[very thick](5,5) to (5,4);
        \draw[very thick](5,4) to (2.8,4);
        \draw[very thick](2.2,4) to (1,4);
        \draw[very thick](1,4) to (1,3);
        \draw[very thick](2.2,3) to (1,3);
        \draw[very thick](2.8,3) to (5,3);
        \draw[very thick](5,2.8) to (5,3);
        \draw[very thick](1,2.2) to (1,2);
        \draw[very thick](2.8,2) to (5,2);
        \draw[very thick](2.2,2) to (1,2);
        \draw[very thick](5,1) to (5,2);
        \draw[very thick](1,1) to (2.2,1);
        \draw[very thick](5,1) to (2.8,1);
        \draw[very thick](1,1) to (1,0);
        \draw[very thick](2.8,0) to (5,0);
        \draw[very thick](2.2,0) to (1,0);
        \draw[thick, bend left=30](0,0) to (5,0);
      
        \draw[dotted](0,2) to (0,3);
        \draw[dotted](1,2) to (1,3);
        \draw[dotted](2,2) to (2,3);
        \draw[dotted](3,2) to (3,3);
        \draw[dotted](4,2) to (4,3);
        \draw[dotted](5,2) to (5,3);
        \draw[dotted](2,0) to (3,0);
        \draw[dotted](2,1) to (3,1);
        \draw[dotted](2,2) to (3,2);
        \draw[dotted](2,3) to (3,3);
        \draw[dotted](2,4) to (3,4);
        \draw[dotted](2,5) to (3,5);
  
      \end{tikzpicture}
      \subcaption{$n$ odd}\label{fig:oddout}
    \end{minipage}
    \caption{Hamiltonian cycle on $C_{2m+1}\times P_n$}
  \end{figure}
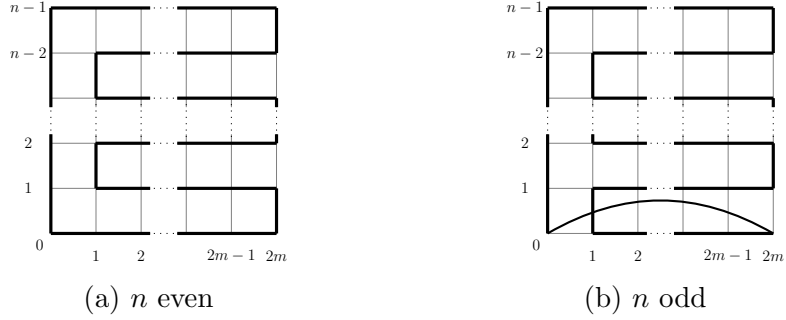
  Considering the vertices of the graph as lattice points on the $xy$-plane, the method for constructing $E_G$ is as follows.
  \begin{itemize}
    \item[1] Start at $(0,0)$, then increase the $y$-coordinate by $1$ to reach $(0,n-1)$.
    \item[2] Increase the $x$-coordinate by $1$ to reach $(2m,n-1)$.
    \item[3] Decrease the $y$-coordinate by $1$ to reach $(2m,n-2)$, then decrease the $x$-coordinate by $1$ to reach $(1,n-2)$.
    \item[4] Decrease the $y$-coordinate by $1$ to reach $(1,n-3)$, then increase the $x$-coordinate by $1$ to reach $(2m,n-3)$.
    \item[5] Repeat steps (3) and (4) as long as the $y$-coordinate is non-negative.
    \item[6] The final point reached will be either $(1,0)$ or $(2m,0)$, ensuring that we can eventually reach $(0,0)$.
  \end{itemize}
  If we replace $(0,n-1)$ with $(0,k-1)$ in step 1, we can construct $C_{k(2m+1)}$ in the manner shown in \Cref{fig:kHamilton}.
Therefore, for any integer $k$ such that $1\leq k\leq n$, we have $C_{k(2m+1)}\subset C_{2m+1}\times P_n$.

  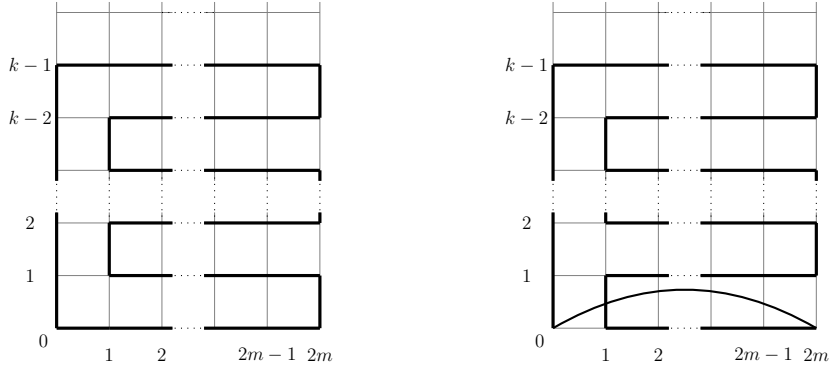
\begin{figure}[H]
    \begin{minipage}[b]{0.4\columnwidth}
      \centering
      \begin{tikzpicture}[scale=0.7]
        \draw[help lines] (0,0) grid (2.2,2.2);
        \draw[help lines] (0,2.8) grid (2.2,6.2);
        \draw[help lines] (2.8,0) grid (5,2.2);
        \draw[help lines] (2.8,2.8) grid (5,6.2);
          
        \node[scale=0.6] ('0') at (-0.25,-0.25) {$0$};
        \node[scale=0.6] ('1') at (1,-0.5) {$1$};
        \node[scale=0.6] ('2') at (2,-0.5) {$2$};
        \node[scale=0.6] ('2m-2') at (3,-0.5) {};
        \node[scale=0.6] ('2m-1') at (4,-0.5) {$2m-1$};
        \node[scale=0.6] ('2m') at (5,-0.5) {$2m$};
          
        \node[scale=0.6] ('v1') at (-0.5,1) {$1$};
        \node[scale=0.6] ('v2') at (-0.5,2) {$2$};
        \node[scale=0.6] ('vk-3') at (-0.5,3) {};
        \node[scale=0.6] ('vk-2') at (-0.5,4) {$k-2$};
        \node[scale=0.6] ('vk-1') at (-0.5,5) {$k-1$};

        \draw[very thick](0,0) to (0,2.2);
        \draw[very thick](0,2.8) to (0,5);
        \draw[very thick](0,5) to (2.2,5);
        \draw[very thick](2.8,5) to (5,5);
        \draw[very thick](5,5) to (5,4);
        \draw[very thick](5,4) to (2.8,4);
        \draw[very thick](2.2,4) to (1,4);
        \draw[very thick](1,4) to (1,3);
        \draw[very thick](2.2,3) to (1,3);
        \draw[very thick](2.8,3) to (5,3);
        \draw[very thick](5,2.8) to (5,3);
        \draw[very thick](5,2.2) to (5,2);
        \draw[very thick](2.8,2) to (5,2);
        \draw[very thick](2.2,2) to (1,2);
        \draw[very thick](1,1) to (1,2);
        \draw[very thick](1,1) to (2.2,1);
        \draw[very thick](5,1) to (2.8,1);
        \draw[very thick](5,1) to (5,0);
        \draw[very thick](2.8,0) to (5,0);
        \draw[very thick](2.2,0) to (1,0);
  
        \draw[very thick](0,0) to (1,0);
      
        \draw[dotted](0,2) to (0,3);
        \draw[dotted](1,2) to (1,3);
        \draw[dotted](2,2) to (2,3);
        \draw[dotted](3,2) to (3,3);
        \draw[dotted](4,2) to (4,3);
        \draw[dotted](5,2) to (5,3);
        \draw[dotted](2,0) to (3,0);
        \draw[dotted](2,1) to (3,1);
        \draw[dotted](2,2) to (3,2);
        \draw[dotted](2,3) to (3,3);
        \draw[dotted](2,4) to (3,4);
        \draw[dotted](2,5) to (3,5);
        \draw[dotted](2,6) to (3,6);
      \end{tikzpicture}  
  
    \end{minipage}
    \hspace{0.1\columnwidth}
    \begin{minipage}[b]{0.4\columnwidth}
      \centering
      \begin{tikzpicture}[scale=0.7]
        \draw[help lines] (0,0) grid (2.2,2.2);
        \draw[help lines] (0,2.8) grid (2.2,6.2);
        \draw[help lines] (2.8,0) grid (5,2.2);
        \draw[help lines] (2.8,2.8) grid (5,6.2);
          
        \node[scale=0.6] ('0') at (-0.25,-0.25) {$0$};
        \node[scale=0.6] ('1') at (1,-0.5) {$1$};
        \node[scale=0.6] ('2') at (2,-0.5) {$2$};
        \node[scale=0.6] ('2m-2') at (3,-0.5) {};
        \node[scale=0.6] ('2m-1') at (4,-0.5) {$2m-1$};
        \node[scale=0.6] ('2m') at (5,-0.5) {$2m$};
          
        \node[scale=0.6] ('v1') at (-0.5,1) {$1$};
        \node[scale=0.6] ('v2') at (-0.5,2) {$2$};
        \node[scale=0.6] ('vk-3') at (-0.5,3) {};
        \node[scale=0.6] ('vk-2') at (-0.5,4) {$k-2$};
        \node[scale=0.6] ('vk-1') at (-0.5,5) {$k-1$};

        \draw[very thick](0,0) to (0,2.2);
        \draw[very thick](0,2.8) to (0,5);
        \draw[very thick](0,5) to (2.2,5);
        \draw[very thick](2.8,5) to (5,5);
        \draw[very thick](5,5) to (5,4);
        \draw[very thick](5,4) to (2.8,4);
        \draw[very thick](2.2,4) to (1,4);
        \draw[very thick](1,4) to (1,3);
        \draw[very thick](2.2,3) to (1,3);
        \draw[very thick](2.8,3) to (5,3);
        \draw[very thick](5,2.8) to (5,3);
        \draw[very thick](1,2.2) to (1,2);
        \draw[very thick](2.8,2) to (5,2);
        \draw[very thick](2.2,2) to (1,2);
        \draw[very thick](5,1) to (5,2);
        \draw[very thick](1,1) to (2.2,1);
        \draw[very thick](5,1) to (2.8,1);
        \draw[very thick](1,1) to (1,0);
        \draw[very thick](2.8,0) to (5,0);
        \draw[very thick](2.2,0) to (1,0);
        \draw[thick, bend left=30](0,0) to (5,0);
      
        \draw[dotted](0,2) to (0,3);
        \draw[dotted](1,2) to (1,3);
        \draw[dotted](2,2) to (2,3);
        \draw[dotted](3,2) to (3,3);
        \draw[dotted](4,2) to (4,3);
        \draw[dotted](5,2) to (5,3);
        \draw[dotted](2,0) to (3,0);
        \draw[dotted](2,1) to (3,1);
        \draw[dotted](2,2) to (3,2);
        \draw[dotted](2,3) to (3,3);
        \draw[dotted](2,4) to (3,4);
        \draw[dotted](2,5) to (3,5);
        \draw[dotted](2,6) to (3,6);
      \end{tikzpicture}
  
    \end{minipage}
    \caption{$C_{k(2m+1)}$ in $C_{2m+1}\times P_n$}\label{fig:kHamilton}
  \end{figure}

Next, let $x$ and $k$ be integers such that $1\leq x\leq 2m$ and $1\leq k\leq n-1$.
We will now demonstrate that $C_{k(2m+1)+x}\subset C_{2m+1}\times P_n$.
We construct $C_{k(2m+1)+x}$ based on the construction of $C_{k(2m+1)}$.
Let $\alpha$ be any vertex on $C_{2m+1}$.
We consider selecting edges up to $(\alpha,1)$, where $\alpha$ is any integer such that $0<\alpha\leq 2m$.
\begin{itemize}
  \item[1] Start at $(0,0)$, then increase the $y$-coordinate by $1$ to reach $(0,k)$.
  \item[2] Increase the $x$-coordinate by $1$ to reach $(2m,k)$.
  \item[3] Decrease the $y$-coordinate by $1$ to reach $(2m,k-1)$, then decrease the $x$-coordinate by $1$ to reach $(1,k-1)$.
  \item[4] Decrease the $y$-coordinate by $1$ to reach $(1,k-2)$, then increase the $x$-coordinate by $1$ to reach $(2m,k-2)$.
  \item[5] Repeat steps (3) and (4) as long as the coordinates are $(\alpha,1)$.
  \item[6] Decrease the $y$-coordinate by $1$ to reach $(\alpha,0)$, where $\alpha\neq 0$.
  \item[7] If $k$ is even, then decrease the $x$-coordinate by $1$ to reach $(0,0)$. Otherwise, increase the $x$-coordinate by $1$ to reach $(0,0)$.
\end{itemize}
This construction is shown in \Cref{fig:xodd}.

When $k$ is even, the graph in which all edges marked with bold lines are included corresponds to a cycle, 
and this cycle passes through 
\[(k+1)(2m+1)-2(2m-\alpha)=k(2m+1)+2(\alpha-m)+1\] 
vertices.
Therefore, if $\alpha=m+i$, then this graph is $C_{k(2m+1)+2i+1}$.
Since $\alpha\leq 2m$, we have $i\leq m$.
Hence, we can choose any $i$ such that $0\leq i\leq m$.
Thus, we can construct $C_{k(2m+1)+x}$ when $x$ is odd and $k$ is even.

When $k$ is odd, the graph is also a cycle that passes through $(k+1)(2m+1)-2(\alpha -1)$ vertices. 
If $\alpha=m+1-i$, then it corresponds to $C_{k(2m+1)+2i+1}$ in a similar manner.
Since $0<\alpha$, we can choose any $i$ such that $0\leq i\leq m$.
Thus, we can conclude that $C_{k(2m+1)+x}\subset C_{2m+1}\times P_n$ holds when both $x$ and $k$ are odd.
In summary, we can state for any $k$ that $C_{k(2m+1)+x}\subset C_{2m+1}\times P_n$ when $x$ is odd.

Now, consider when $x$ is even.
Let $\alpha$ be any integer such that $0\leq\alpha\leq m$.
We modify steps 1 and 2 of the previous method in the following manner.
\begin{itemize}
  \item[1] Start at $(0,0)$, then increase the $y$-coordinate by $1$ to reach $(0,k)$.
  \item[2] Choose edges as follows:
  \begin{align*}
    &\{(0,k),(1,k)\},\{(1,k),(1,k-1)\},\{(1,k-1),(2,k-1)\},\{(2,k-1),(2,k)\}, \ldots\\
    &\{(2\alpha -2,k),(2\alpha -1,k)\},\{(2\alpha -1,k),(2\alpha -1,k-1)\},\{(2\alpha-1,k-1),(2\alpha,k-1)\},\ldots,\\
    &\{(2m-1,k-1),(2m,k-1)\}.
  \end{align*}
  \item[3] Decrease the $y$-coordinate by $1$ to reach $(2m,k-2)$, then decrease the $x$-coordinate by $1$ to reach $(1,k-2)$.
  \item[4] Decrease the $y$-coordinate by $1$ to reach $(1,k-3)$, then increase the $x$-coordinate by $1$ to reach $(2m,k-3)$.
  \item[5] Repeat steps (3) and (4) as long as the $y$-coordinate is non-negative.
  \item[6] The final point reached will be either $(1,0)$ or $(2m,0)$, ensuring that we can eventually reach $(0,0)$.
\end{itemize}
This construction is illustrated in \Cref{fig:xeven}.

\begin{figure}[H]
  \centering
  \begin{minipage}[b]{0.4\textwidth}

    \begin{tikzpicture}[xscale=0.65,yscale=0.65]
      \draw[help lines] (0,0) grid (2.2,2.2);
      \draw[help lines] (0,2.8) grid (2.2,6.2);
      \draw[help lines] (2.8,0) grid (5.2,2.2);
      \draw[help lines] (2.8,2.8) grid (5.2,6.2);
      \draw[help lines] (5.8,0) grid (8,2.2);
      \draw[help lines] (5.8,2.8) grid (8,6.2);
    
      \node[scale=0.6] ('0') at (-0.25,-0.25) {$0$};
      \node[scale=0.6] ('1') at (1,-0.5) {$1$};
      \node[scale=0.6] ('2') at (2,-0.5) {$2$};
      % \node[scale=0.7] ('2m-2') at (6,-0.5) {2m-2};
      \node[scale=0.6] ('2m-1') at (7,-0.5) {$2m-1$};
      \node[scale=0.6] ('2m') at (8,-0.5) {$2m$};
      \node[scale=0.6] ('alpha') at (4,-0.5) {$\alpha$};
    
      \node[scale=0.6] ('v1') at (-0.5,1) {$1$};
      \node[scale=0.6] ('v2') at (-0.5,2) {$2$};
      \node[scale=0.6] ('vn-3') at (-0.5,3) {$k-2$};
      \node[scale=0.6] ('vn-2') at (-0.5,4) {$k-1$};
      \node[scale=0.6] ('vn-1') at (-0.5,5) {$k$};
    
      \draw[very thick](0,0) to (0,2.2);
      \draw[very thick](0,2.8) to (0,5);
      \draw[very thick](0,5) to (2.2,5);
      \draw[very thick](2.8,5) to (5,5);
      \draw[very thick](5,5) to (5.2,5);
      \draw[very thick](5.8,5) to (8,5);
      \draw[very thick](8,4) to (8,5);
      \draw[very thick](5,4) to (2.8,4);
      \draw[very thick](2.2,4) to (1,4);
      \draw[very thick](1,4) to (1,3);
      \draw[very thick](2.2,3) to (1,3);
      \draw[very thick](2.8,3) to (5,3);
      \draw[very thick](8,2.8) to (8,3);
      \draw[very thick](8,2.2) to (8,2);
      \draw[very thick](2.8,2) to (5,2);
      \draw[very thick](2.2,2) to (1,2);
      \draw[very thick](1,1) to (1,2);
      \draw[very thick](1,1) to (2.2,1);
      \draw[very thick](4,1) to (2.8,1);
      \draw[very thick](4,1) to (4,0);
      \draw[very thick](2.8,0) to (4,0);
      \draw[very thick](2.2,0) to (1,0);
      \draw[very thick](5.8,4) to (8,4);
      \draw[very thick](5.8,3) to (8,3);
      \draw[very thick](5.8,2) to (8,2);
    
      \draw[very thick](0,0) to (1,0);
    
      \draw[dotted](0,2) to (0,3);
      \draw[dotted](1,2) to (1,3);
      \draw[dotted](2,2) to (2,3);
      \draw[dotted](3,2) to (3,3);
      \draw[dotted](4,2) to (4,3);
      \draw[dotted](5,2) to (5,3);
      \draw[dotted](6,2) to (6,3);
      \draw[dotted](7,2) to (7,3);
      \draw[dotted](8,2) to (8,3);
      \draw[dotted](2,0) to (3,0);
      \draw[dotted](2,1) to (3,1);
      \draw[dotted](2,2) to (3,2);
      \draw[dotted](2,3) to (3,3);
      \draw[dotted](2,4) to (3,4);
      \draw[dotted](2,5) to (3,5);
      \draw[dotted](2,6) to (3,6);
      \draw[dotted](5,0) to (6,0);
      \draw[dotted](5,1) to (6,1);
      \draw[dotted](5,2) to (6,2);
      \draw[dotted](5,3) to (6,3);
      \draw[dotted](5,4) to (6,4);
      \draw[dotted](5,5) to (6,5);
      \draw[dotted](5,6) to (6,6);
    
    \end{tikzpicture}
    \subcaption{$k$ even}\label{fig:xoddkeven}
  \end{minipage}
  \hspace{0.08\textwidth}
  \begin{minipage}[b]{0.4\textwidth}
    \begin{tikzpicture}[xscale=0.65,yscale=0.65]
      \draw[help lines] (0,0) grid (2.2,2.2);
      \draw[help lines] (0,2.8) grid (2.2,6.2);
      \draw[help lines] (2.8,0) grid (5.2,2.2);
      \draw[help lines] (2.8,2.8) grid (5.2,6.2);
      \draw[help lines] (5.8,0) grid (8,2.2);
      \draw[help lines] (5.8,2.8) grid (8,6.2);
    
      \node[scale=0.6] ('0') at (-0.25,-0.25) {$0$};
      \node[scale=0.6] ('1') at (1,-0.5) {$1$};
      \node[scale=0.6] ('2') at (2,-0.5) {$2$};
      % \node[scale=0.7] ('2m-2') at (6,-0.5) {2m-2};
      \node[scale=0.6] ('2m-1') at (7,-0.5) {$2m-1$};
      \node[scale=0.6] ('2m') at (8,-0.5) {$2m$};
      \node[scale=0.6] ('alpha') at (4,-0.5) {$\alpha$};
    
      \node[scale=0.6] ('v1') at (-0.5,1) {$1$};
      \node[scale=0.6] ('v2') at (-0.5,2) {$2$};
      \node[scale=0.6] ('vn-3') at (-0.5,3) {$k-2$};
      \node[scale=0.6] ('vn-2') at (-0.5,4) {$k-1$};
      \node[scale=0.6] ('vn-1') at (-0.5,5) {$k$};
    
      \draw[very thick](0,0) to (0,2.2);
      \draw[very thick](0,2.8) to (0,5);
      \draw[very thick](0,5) to (2.2,5);
      \draw[very thick](2.8,5) to (5,5);
      \draw[very thick](5,5) to (5.2,5);
      \draw[very thick](5.8,5) to (8,5);
      \draw[very thick](8,4) to (8,5);
      \draw[very thick](5,4) to (2.8,4);
      \draw[very thick](2.2,4) to (1,4);
      \draw[very thick](1,4) to (1,3);
      \draw[very thick](2.2,3) to (1,3);
      \draw[very thick](2.8,3) to (5,3);
      \draw[very thick](8,2.8) to (8,3);
      \draw[very thick](1,2.2) to (1,2);
      \draw[very thick](2.8,2) to (5,2);
      \draw[very thick](2.2,2) to (1,2);
      \draw[very thick](8,1) to (8,2);
      \draw[very thick](8,1) to (5.8,1);
      \draw[very thick](5.2,1) to (4,1);
      \draw[very thick](4,1) to (4,0);
      \draw[very thick](4,0) to (5.2,0);
      \draw[very thick](5.8,0) to (8,0);
      \draw[very thick](5.8,4) to (8,4);
      \draw[very thick](5.8,3) to (8,3);
      \draw[very thick](5.8,2) to (8,2);
      \draw[thick, bend left=20](0,0) to (8,0);
    
      \draw[dotted](0,2) to (0,3);
      \draw[dotted](1,2) to (1,3);
      \draw[dotted](2,2) to (2,3);
      \draw[dotted](3,2) to (3,3);
      \draw[dotted](4,2) to (4,3);
      \draw[dotted](5,2) to (5,3);
      \draw[dotted](6,2) to (6,3);
      \draw[dotted](7,2) to (7,3);
      \draw[dotted](8,2) to (8,3);
      \draw[dotted](2,0) to (3,0);
      \draw[dotted](2,1) to (3,1);
      \draw[dotted](2,2) to (3,2);
      \draw[dotted](2,3) to (3,3);
      \draw[dotted](2,4) to (3,4);
      \draw[dotted](2,5) to (3,5);
      \draw[dotted](5,0) to (6,0);
      \draw[dotted](5,1) to (6,1);
      \draw[dotted](5,2) to (6,2);
      \draw[dotted](5,3) to (6,3);
      \draw[dotted](5,4) to (6,4);
      \draw[dotted](5,5) to (6,5);
    \end{tikzpicture}
    \subcaption{$k$ odd}\label{fig:xoddkodd}
  \end{minipage}
\caption{Construction when $x$ is odd}\label{fig:xodd}

\end{figure}
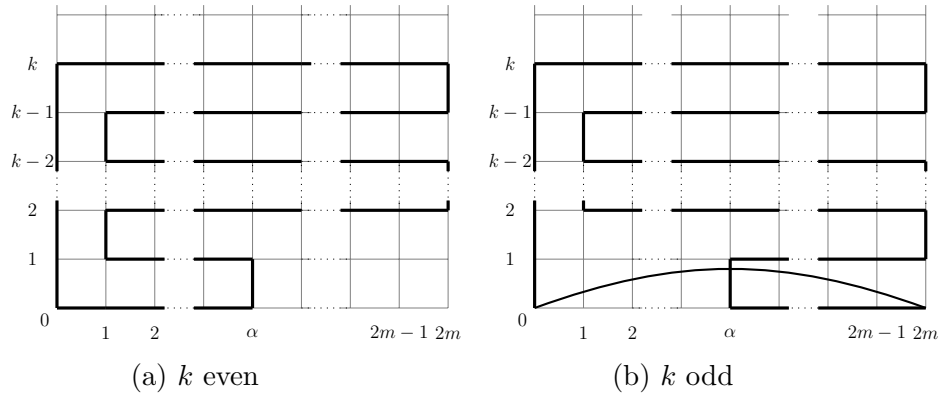

\begin{figure}[H]
\centering
\begin{minipage}[b]{0.4\textwidth}
  \begin{tikzpicture}[xscale=0.65,yscale=0.65]
    \draw[help lines] (0,0) grid (2.2,2.2);
    \draw[help lines] (0,2.8) grid (2.2,6.2);
    \draw[help lines] (2.8,0) grid (5.2,2.2);
    \draw[help lines] (2.8,2.8) grid (5.2,6.2);
    \draw[help lines] (5.8,0) grid (8,2.2);
    \draw[help lines] (5.8,2.8) grid (8,6.2);
  
    \node[scale=0.6] ('0') at (-0.25,-0.25) {$0$};
    \node[scale=0.6] ('1') at (1,-0.5) {$1$};
    \node[scale=0.6] ('2') at (2,-0.5) {$2$};
    \node[scale=0.6] ('2m-2') at (6,-0.5) {};
    \node[scale=0.6] ('2m-1') at (7,-0.5) {$2m-1$};
    \node[scale=0.6] ('2m') at (8,-0.5) {$2m$};
    \node[scale=0.6] ('alpha') at (4,-0.5) {$2\alpha -1$};
  
    \node[scale=0.6] ('v1') at (-0.5,1) {$1$};
    \node[scale=0.6] ('v2') at (-0.5,2) {$2$};
    \node[scale=0.6] ('vn-3') at (-0.5,3) {$k-2$};
    \node[scale=0.6] ('vn-2') at (-0.5,4) {$k-1$};
    \node[scale=0.6] ('vn-1') at (-0.5,5) {$k$};
  
    \draw[very thick](0,0) to (0,2.2);
    \draw[very thick](0,2.8) to (0,5);
    \draw[very thick](0,5) to (1,5);
    \draw[very thick](1,5) to (1,4);
    \draw[very thick](1,4) to (2,4);
    \draw[very thick](2,4) to (2,5);
    \draw[very thick](2.2,5) to (2,5);
    \draw[very thick](3,4) to (2.8,4);
    \draw[very thick](3,5) to (3,4);
    \draw[very thick](3,5) to (4,5);
    \draw[very thick](4,5) to (4,4);
    \draw[very thick](5.2,4) to (4,4);
    \draw[very thick](5.8,4) to (8,4);
  
    \draw[very thick](8,4) to (8,3);
    \draw[very thick](2.2,3) to (1,3);
    \draw[very thick](2.8,3) to (5,3);
    \draw[very thick](1,2.8) to (1,3);
    \draw[very thick](8,2.2) to (8,2);
    \draw[very thick](2.8,2) to (5,2);
    \draw[very thick](2.2,2) to (1,2);
    \draw[very thick](1,1) to (1,2);
    \draw[very thick](1,1) to (2.2,1);
    \draw[very thick](5,1) to (2.8,1);
    \draw[very thick](8,1) to (8,0);
    \draw[very thick](2.8,0) to (5,0);
    \draw[very thick](2.2,0) to (1,0);
    \draw[very thick](5.8,4) to (8,4);
    \draw[very thick](5.8,3) to (8,3);
    \draw[very thick](5.8,2) to (8,2);
    \draw[very thick](5.8,1) to (8,1);
    \draw[very thick](5.8,0) to (8,0);
    \draw[very thick](5.2,2) to (5,2);
    \draw[very thick](5.2,1) to (5,1);
    \draw[very thick](5.2,0) to (5,0);
  
    \draw[very thick](0,0) to (1,0);
  
    \draw[dotted](0,2) to (0,3);
    \draw[dotted](1,2) to (1,3);
    \draw[dotted](2,2) to (2,3);
    \draw[dotted](3,2) to (3,3);
    \draw[dotted](4,2) to (4,3);
    \draw[dotted](5,2) to (5,3);
    \draw[dotted](6,2) to (6,3);
    \draw[dotted](7,2) to (7,3);
    \draw[dotted](8,2) to (8,3);
    \draw[dotted](2,0) to (3,0);
    \draw[dotted](2,1) to (3,1);
    \draw[dotted](2,2) to (3,2);
    \draw[dotted](2,3) to (3,3);
    \draw[dotted](2,4) to (3,4);
    \draw[dotted](2,5) to (3,5);
    \draw[dotted](5,0) to (6,0);
    \draw[dotted](5,1) to (6,1);
    \draw[dotted](5,2) to (6,2);
    \draw[dotted](5,3) to (6,3);
    \draw[dotted](5,4) to (6,4);
    \draw[dotted](5,5) to (6,5);
  \end{tikzpicture}
  \subcaption{$k$ even}
\end{minipage}
\hspace{0.08\textwidth}
\begin{minipage}[b]{0.4\textwidth}
  \begin{tikzpicture}[xscale=0.65,yscale=0.65]
    \draw[help lines] (0,0) grid (2.2,2.2);
    \draw[help lines] (0,2.8) grid (2.2,6.2);
    \draw[help lines] (2.8,0) grid (5.2,2.2);
    \draw[help lines] (2.8,2.8) grid (5.2,6.2);
    \draw[help lines] (5.8,0) grid (8,2.2);
    \draw[help lines] (5.8,2.8) grid (8,6.2);
  
    \node[scale=0.6] ('0') at (-0.25,-0.25) {$0$};
    \node[scale=0.6] ('1') at (1,-0.5) {$1$};
    \node[scale=0.6] ('2') at (2,-0.5) {$2$};
    \node[scale=0.6] ('2m-2') at (6,-0.5) {};
    \node[scale=0.6] ('2m-1') at (7,-0.5) {$2m-1$};
    \node[scale=0.6] ('2m') at (8,-0.5) {$2m$};
    \node[scale=0.6] ('alpha') at (4,-0.5) {$2\alpha -1$};
  
    \node[scale=0.6] ('v1') at (-0.5,1) {$1$};
    \node[scale=0.6] ('v2') at (-0.5,2) {$2$};
    \node[scale=0.6] ('vn-3') at (-0.5,3) {$k-2$};
    \node[scale=0.6] ('vn-2') at (-0.5,4) {$k-1$};
    \node[scale=0.6] ('vn-1') at (-0.5,5) {$k$};
  
    \draw[very thick](0,0) to (0,2.2);
    \draw[very thick](0,2.8) to (0,5);
    \draw[very thick](0,5) to (1,5);
    \draw[very thick](1,5) to (1,4);
    \draw[very thick](1,4) to (2,4);
    \draw[very thick](2,4) to (2,5);
    \draw[very thick](2.2,5) to (2,5);
    \draw[very thick](3,4) to (2.8,4);
    \draw[very thick](3,5) to (3,4);
    \draw[very thick](3,5) to (4,5);
    \draw[very thick](4,5) to (4,4);
    \draw[very thick](5.2,4) to (4,4);
    \draw[very thick](5.8,4) to (8,4);
  
    \draw[very thick](8,4) to (8,3);
    \draw[very thick](2.2,3) to (1,3);
    \draw[very thick](2.8,3) to (5,3);
    \draw[very thick](1,2.8) to (1,3);
    \draw[very thick](1,2.2) to (1,2);
    \draw[very thick](2.8,2) to (5,2);
    \draw[very thick](2.2,2) to (1,2);
    \draw[very thick](8,1) to (8,2);
    \draw[very thick](1,1) to (2.2,1);
    \draw[very thick](5,1) to (2.8,1);
    \draw[very thick](1,1) to (1,0);
    \draw[very thick](2.8,0) to (5,0);
    \draw[very thick](2.2,0) to (1,0);
    \draw[very thick](5.8,4) to (8,4);
    \draw[very thick](5.8,3) to (8,3);
    \draw[very thick](5.8,2) to (8,2);
    \draw[very thick](5.8,1) to (8,1);
    \draw[very thick](5.8,0) to (8,0);
    \draw[very thick](5.2,2) to (5,2);
    \draw[very thick](5.2,1) to (5,1);
    \draw[very thick](5.2,0) to (5,0);
  
    \draw[thick, bend left=20](0,0) to (8,0);
  
    \draw[dotted](0,2) to (0,3);
    \draw[dotted](1,2) to (1,3);
    \draw[dotted](2,2) to (2,3);
    \draw[dotted](3,2) to (3,3);
    \draw[dotted](4,2) to (4,3);
    \draw[dotted](5,2) to (5,3);
    \draw[dotted](6,2) to (6,3);
    \draw[dotted](7,2) to (7,3);
    \draw[dotted](8,2) to (8,3);
    \draw[dotted](2,0) to (3,0);
    \draw[dotted](2,1) to (3,1);
    \draw[dotted](2,2) to (3,2);
    \draw[dotted](2,3) to (3,3);
    \draw[dotted](2,4) to (3,4);
    \draw[dotted](2,5) to (3,5);
    \draw[dotted](5,0) to (6,0);
    \draw[dotted](5,1) to (6,1);
    \draw[dotted](5,2) to (6,2);
    \draw[dotted](5,3) to (6,3);
    \draw[dotted](5,4) to (6,4);
    \draw[dotted](5,5) to (6,5);
  \end{tikzpicture}
  \subcaption{$k$ odd}
\end{minipage}

\caption{Construction when $x$ is even}\label{fig:xeven}
\end{figure}
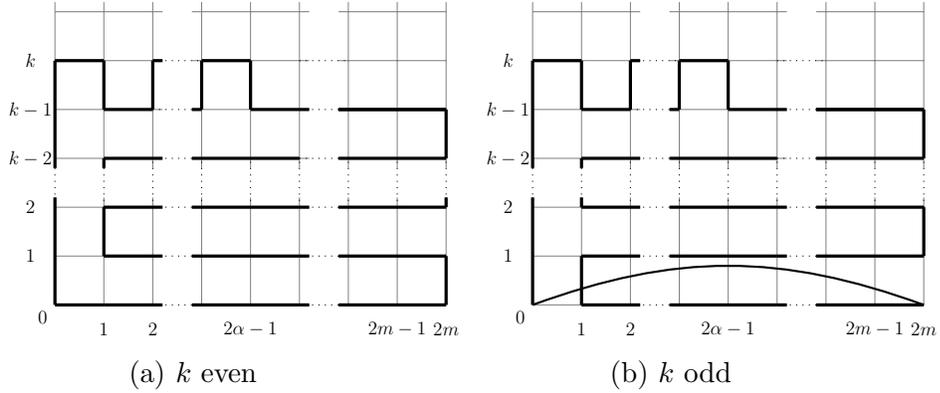

This graph also forms a cycle, which
passes through $k(2m+1)+2\alpha$ vertices.
Taking into account the conditions of $\alpha$, we can observe that for any $x$ that is even, $C_{k(2m+1)+x}\subset C_{2m+1}\times P_n$ holds.

Thus, we can conclude that for any $k$ and $x$ such that $1\leq k\leq n-1$ and $1\leq x\leq 2m$
\[
  C_{k(2m+1)+x}\subset C_{2m+1}\times P_n.
\]
We have already established that
\[
  C_{k(2m+1)}\subset C_{2m+1}\times P_n.
\]
Therefore, the proof is complete.
\end{proof}

The following corollary, which is almost equivalent to \Cref{lem:carPartPanc}, can also be obtained.

\begin{cor}\label{cor:2}
  Let $G_1$ be a Hamiltonian graph with order $2m+1$, and let $G_2$ be a semi-Hamiltonian graph with order $n$, where $m$ and $n$ are any positive integers. 
  Then $G_1\times G_2$ is a Hamiltonian graph. 
  Additionally, for any integer $k$ such that $2m+1\leq k\leq n(2m+1)$, we have $C_k\subset G_1\times G_2$.
\end{cor}

\begin{proof}
  By the definitions of a Hamiltonian graph and a semi-Hamiltonian graph, $C_{2m+1}\times P_n\subset G_1\times G_2$. 
  Hence, apply \Cref{lem:carPartPanc}, and the proof complete.
\end{proof}

Next, we prove ~\Cref{thm:carPanc} using ~\Cref{cor:2}.

\begin{proof}[Proofs of \Cref{thm:carPanc}]
  Let the orders of $G_1$ and $G_2$ be $2m+1$ and $n$, respectively.
  Since $G_1$ is pancyclic, it is also a Hamiltonian graph.
  Therefore, by \Cref{cor:2}, we know that for any integer $k$ such that $2m+1\leq k\leq n(2m+1)$, we have $C_k \subset G_1\times G_2$. 
Additionally, the pancyclicity of $G_1$ implies that for any integer $k$ such that $3\leq k\leq 2m+1$, we have $C_k\subset G_1\subset G_1\times G_2$.

These results demonstrate that $G_1\times G_2$ is pancyclic, thus completing the proof.
\end{proof}

\section{Proof of \Cref{thm:main}}\label{sec:pay}

In this section, we provide proofs for \Cref{thm:main} and \Cref{cor:index}.

\begin{proof}[Proof of \Cref{thm:main}]
  ~\Cref{cor:paPanc} states that a Paley graph $P(q)$ is pancyclic when $q$ is a prime number and not equal to $5$. 
  Therefore, we need to consider the case where $q$ is not a prime number, which means $q=q_0^n$, where $q_0$ is a prime number and $n$ is an integer greater than $1$.

We can view the elements of $\FF_{q}=\FF_{q_0^n}$ as an $n$-dimensional vector space composed of polynomials with coefficients from $\FF_{q_0}$. 
Since $(\FF_{q_0^n}^*)^2$ can generate $\FF_{q_0^n}$, there exist basis vectors of $\FF_{q_0^n}$ that are all in $(\FF_{q_0^n}^*)^2$.
We denote these basis vectors as $f_0,f_1,\ldots,f_{n-1}$.

Let $W_0$ be the subspace of $\FF_{q_0^n}$ generated by $f_0$, and let $W^*_{n-1}$ be the subspace generated by $f_1, f_2, \ldots, f_{n-1}$. 
We define $G_{W_0}$ as the induced subgraph of $P(q)$ on $W_0$, and $G_{W^*_{n-1}}$ as the induced subgraph on $W^*_{n-1}$.
We aim to show that
\[
  G_{W_0}\times G_{W^*_{n-1}} \subset P(q).
\]
Now, since any element $\alpha$ in $\mathbb{F}_{q_0^n}$ can be expressed as $\alpha=\alpha_0+\alpha_{n-1}$ using $\alpha_0\in W_0$ and $\alpha_{n-1} \in W^*_{n-1}$, 
we can regard $\mathbb{F}_{q_0^n}$ as $W_0\times W^*_{n-1}$.
Therefore, the vertex set of $G_{W_0}\times G_{W^*_{n-1}}$ is included in the vertex set of $P(q)$.

Let $(\alpha_0,\alpha_{n-1})$ and $(\beta_0,\beta_{n-1})$ be vertices in $G_{W_0}\times G_{W^*_{n-1}}$. 
By the definitions of the Cartesian product and the Paley graph, we can state that
\begin{align*}
  (\alpha_0,\alpha_{n-1})\sim(\beta_0,\beta_{n-1}) 
  &\Longrightarrow  ((\alpha_0=\beta_0)\land(\alpha_{n-1}\sim \beta_{n-1}))\\
  &\quad\quad\quad\quad\lor((\alpha_{0}\sim \beta_{0})\land(\alpha_{n-1}=\beta_{n-1}))\\
  &\Longrightarrow ((\alpha_0=\beta_0)\land(\alpha_{n-1}-\beta_{n-1}\in (\FF_{q_0^n}^*)^2))\\
  &\quad\quad\quad\quad\lor((\alpha_{0}-\beta_{0}\in (\FF_{q_0^n}^*)^2)\land(\alpha_{n-1}=\beta_{n-1}))\\
  &\Longrightarrow (\alpha_0+\alpha_{n-1})-(\beta_0+\beta_{n-1}) \in (\FF_{q_0^n}^*)^2.
\end{align*}
We can equate $\alpha_0+\alpha_{n-1}=\alpha$ and $\beta_0+\beta_{n-1}=\beta$ with the vertices of $P(q)$. 
If $\alpha - \beta \in  (\FF_{q_0^n}^*)^2$, then $\alpha \sim \beta$.
Therefore, we can write
\[
  (\alpha_0,\alpha_{n-1})\sim(\beta_0,\beta_{n-1}) \Longrightarrow \alpha \sim \beta.
\]
This means that any edges in $G_{W_0}\times G_{W^*_{n-1}}$ are also present in $P(q)$. 
Since $P(q)$ contains both the vertex set and the edge set of $G_{W_0}\times G_{W^*_{n-1}}$, we can conclude that $G_{W_0}\times G_{W^*_{n-1}}\subset P(q)$.
Note that $G_{W_0}\times G_{W^*_{n-1}}$ is a spanning subgraph because $|G_{W^*_{n-1}}|=q_0^{n-1},|G_{W_0}|=q_0$.
To apply ~\Cref{thm:carPanc} and ~\Cref{cor:2} to $G_{W_0}\times G_{W^*_{n-1}}$, 
we claim that $G_{W_0}$ is either pancyclic or a Hamiltonian graph, and that $G_{W^*_{n-1}}$ is a semi-Hamiltonian graph.

First, consider $G_{W^*_{n-1}}$.
For any integer $i$ such that $1\leq i\leq n-1$, let $W^*_i$ be the subspace generated by $f_1,f_2,\ldots,f_i$, 
and let $G_{W^*_i}$ be the induced subgraph with vertices from $W^*_i$. 
We claim that $G_{W^*_i}$ is a semi-Hamiltonian graph by induction.

For the base case, when $i=1$, the vertices of $W^*_1$ are 
\[
  \{0,f_1,2f_1,\ldots,(q_0-1)f_1\},
\] 
and since $f_1\in (\FF_q^*)^2$, we can choose the edges
\[
\{(0,f_1),(f_1,2f_1),(2f_1,3f_1),\ldots,((q_0-2)f_1,(q_0-1)f_1),(0,(q_0-1)f_1)\}.  
\]
This is a Hamiltonian cycle of $G_{W^*_1}$.
Therefore, we have shown that $G_{W^*_1}$ is a Hamiltonian graph, implying that $G_{W^*_1}$ is a semi-Hamiltonian graph. 
These edges can be represented as shown in \Cref{fig:hamiX1}.

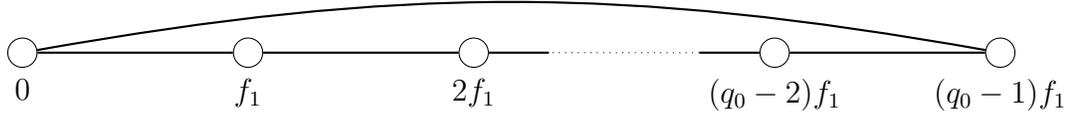
\begin{figure}[H]
  \centering
  \tikzset{my node/.style={circle,inner sep=0pt}}

  \begin{tikzpicture}
    
    \node[fill=white, draw, circle, label=below:$0$] (0) at (0,0) {};
    \node[fill=white, draw, circle, label=below:$f_{1}$] (1) at (3,0) {};
    \node[fill=white, draw, circle, label=below:$2f_{1}$] (2) at (6,0) {};
    \node[fill=white, draw, circle, label=below:$(q_0-2)f_{1}$] (2m-2) at (10,0) {};
    \node[fill=white, draw, circle, label=below:$(q_0-1)f_{1}$] (2m-1) at (13,0) {};

    \draw[thick](0) to (1);
    \draw[thick](2) to (1);
    \draw[thick](2) to (7,0);
    \draw[dotted](9,0) to (7,0);
    \draw[thick](2m-2) to (9,0);
    \draw[thick](2m-2) to (2m-1);
    \draw[thick, bend left=10](0) to (2m-1);

  \end{tikzpicture}
  \caption{Hamiltonian cycle of $f_1$}\label{fig:hamiX1}
\end{figure}

Next, assume that $G_{W^*_i}$ is a semi-Hamiltonian graph for some integer $i$ such that $1\leq i\leq n-2$. 
Now, consider $G_{W^*_{i+1}}$.
Let $X_{i+1}$ be the subspace generated by $f_{i+1}$ alone. 
We also define $G_{X_{i+1}}$ as an induced subgraph with vertices from $X_{i+1}$.
We consider the relationship between $G_{X_{i+1}}\times G_{W^*_i}$ and $G_{W^*_{i+1}}$.
Regarding their vertex sets, 
for any $\alpha \in W^*_{i+1}$, there exist $\alpha_i\in W^*_i$ and $\alpha_{i+1}\in X_{i+1}$ such that $\alpha=\alpha_{i}+\alpha_{i+1}$.
Conversely, for any $\alpha_i\in W^*_i$ and $\alpha_{i+1}\in X_{i+1}$, $\alpha_{i}+\alpha_{i+1}\in W^*_{i+1}$.
Therefore, we have $W^*_{i+1}= X_{i+1}\times W^*_{i}$.
Regarding the set of edges, using the same discussion as in the proof of $G_{W_0}\times G_{W^*_{n-1}}\subset P(q)$, 
for any $(\alpha_i,\alpha_{i+1}),(\beta_i,\beta_{i+1})\in  X_{i+1}\times W^*_{i}$, we claim that
\[
  (\alpha_i,\alpha_{i+1})\sim(\beta_i,\beta_{i+1})\Longrightarrow (\alpha_i+\alpha_{i+1})-(\beta_i+\beta_{i+1})\in (\FF_{q_0^n}^*)^2\Longrightarrow \alpha\sim\beta .
\]
This implies that the edge set of $G_{X_{i+1}}\times G_{W^*_i}$ is contained in the edge set of $G_{W^*_{i+1}}$.
Therefore, $G_{X_{i+1}}\times G_{W^*_i}$ is a spanning subgraph of $G_{W^*_{i+1}}$.

Since $X_{i+1}$ consists of the elements 
\[\{0,f_{i+1},2f_{i+1},\ldots,(q_0-1)f_{i+1}\}\] 
and $f_{i+1}\in (\FF_q^*)^2$,
similar to the case of $W^*_1$, we can choose
\[
\{(0,f_{i+1}),(f_{i+1},2f_{i+1}),(2f_{i+1},3f_{i+1}),\ldots,((q_0-2)f_{i+1},(q_0-1)f_{i+1}),(0,(q_0-1)f_{i+1})\}  
\]
as the edges for $G_{X_{i+1}}$.
This implies that $G_{X_{i+1}}$ is a Hamiltonian graph with order $q_0$, which is a prime number.
Based on the assumption that $G_{W^*_i}$ is a semi-Hamiltonian graph, we can apply \Cref{cor:2} to $G_{X_{i+1}}\times G_{W^*_i}$. 
Hence, we know that $G_{W^*_{i+1}}$ is a semi-Hamiltonian graph.

Therefore, by induction for any $i$ such that $1\leq i\leq n-1$, we can conclude that $G_{W^*_{i}}$ is a semi-Hamiltonian graph.

Next, we demonstrate the pancyclicity of $G_{W_0}$ when $q_0\neq 5$. Since
\[
  W_0=\{0,f_0,2f_0,\ldots,(q_0-1)f_0\},
\]
two vertices $if_0,jf_0\in G_{W_0}$ are adjacent if and only if 
\[
  (j-i)f_0\in (\FF_{q_0^n}^*)^2, 
\]
where $i,j\in \FF_{q_0}, i\neq j$.
Since $f_0\in (\FF_{q_0^n}^*)^2$ and $j-i\in\FF_{q_0}^*$, $(j-i)f_0\in(\FF_{q_0^n}^*)^2$ if and only if 
\[
  j-i\in (\FF_{q_0^n}^*)^2\cap\FF_{q_0}^*. 
\]
Hence, the graph
\begin{align*}
  \phi: G_{W_0}&\mapsto Cay(\FF_{q_0},(\FF_{q_0^n}^*)^2\cap\FF_{q_0}^*)\\
        if_0 &\mapsto i
\end{align*}
is isomorphic.
Next, we show that $Cay(\FF_{q_0},(\FF_{q_0^n}^*)^2\cap\FF_{q_0}^*)$ is pancyclic.
Let 
\[
  H=Cay(\FF_{q_0},(\FF_{q_0^n}^*)^2\cap\FF_{q_0}^*),
\] 
and let $\pm (\FF_{q_0}^*)^2$ be generated by $-1$ and $(\FF_{q_0}^*)^2$.
Since $(\FF_{q_0}^*)^2\subset(\FF_{q_0^n}^*)^2$ and $-1\in (\FF_{q_0^n}^*)^2$, we can write
\[
  \pm(\FF_{q_0}^*)^2 \subset (\FF_{q_0^n}^*)^2\cap\FF_{q_0}^*.
\]
Hence, $Cay(\FF_{q_0},\pm(\FF_{q_0}^*)^2)\subset H$.
Obviously, $Cay(\FF_{q_0},\pm(\FF_{q_0}^*)^2)$ is a spanning subgraph of $H$. 

If $q_0\equiv 1 \pmod{4}$, then 
\[
  \pm (\FF_{q_0}^*)^2=(\FF_{q_0}^*)^2
\] because $-1\in (\FF_{q_0}^*)^2$.
Thus, 
\[
  Cay(\FF_{q_0},\pm (\FF_{q_0}^*)^2)=Cay(\FF_{q_0},(\FF_{q_0}^*)^2)=P(q_0).
\]
Since $q_0$ is a prime number not equal to $5$, \Cref{cor:paPanc} shows that $P(q_0)$ is pancyclic.
Therefore, $Cay(\FF_{q_0},\pm(\FF_{q_0}^*)^2)$ is also pancyclic.

When $q_0\equiv 3 \pmod{4}$, since $-1$ is not in $(\mathbb{F}_{q_0}^{*})^2$,
we have $-\alpha \notin (\mathbb{F}_{q_0}^{*})^2$ for any $\alpha \in (\mathbb{F}_{q_0}^{*})^2$.
This implies that the sets $(\mathbb{F}_{q_0}^{*})^2$ and $-(\mathbb{F}_{q_0}^{})^2$ are disjoint.
Since 
\[
  |(\mathbb{F}_{q_0}^{*})^2|=-\vert(\mathbb{F}_{q_0}^{*})^2\vert = \frac{\vert\mathbb{F}_{q_0}^{*}\vert}{2},
\]
we have 
\[
  \pm (\FF_{q_0}^*)^2 = \mathbb{F}_{q_0}^{*}.
\]
Consequently, 
\[
  Cay(\FF_{q_0},\pm (\FF_{q_0}^*)^2) = Cay(\FF_{q_0},\FF_{q_0}^*) = K_{q_0},
\]
where $K_{q_0}$ is the complete graph of order $q_0$.
Since the complete graph is pancyclic, this is also pancyclic.
Therefore, we can conclude that $Cay(\FF_{q_0},\pm(\FF_{q_0}^*)^2)$ is pancyclic as well.
Since $Cay(\FF_{q_0},\pm(\FF_{q_0}^*)^2)$ is a spanning subgraph of $H$, $H\simeq G_{W_0}$ is also pancyclic.

In summary, we have shown that $G_{W_0}$ is pancyclic when $q_0\neq 5$.
Since $G_{W^*_{n-1}}$ has a Hamiltonian path, and $G_{W_0}$ is pancyclic with an odd order when $q_0\neq 5$, 
we can conclude from \Cref{thm:carPanc} that $G_{W_0}\times G_{W^*_{n-1}}$ is pancyclic.
Furthermore, since we have already established that $G_{W_0}\times G_{W^*_{n-1}}$ is a spanning subgraph of $P(q_0^n)$, 
we can also deduce that $P(q_0^n)$ is pancyclic when $q_0\neq 5$.

Finally, we wish to show that $P(5^n)$ is also pancyclic.
Since 
\[
  W_0=\{0,f_0,2f_0,3f_0,4f_0\}
\] 
and $\pm 1\in (\FF_5^*)^2$, we can choose the edges in $G_{W_0}$ to be
\[
  \{(0,f_{0}),(f_{0},2f_{0}),(2f_{0},3f_{0}),(3f_{0},4f_{0}),(0,4f_{0})\}.  
\]
This shows that $G_{W_0}$ has a Hamiltonian cycle. We also know that $G_{W^*_{n-1}}$ is a semi-Hamiltonian graph. 
Hence, we can apply \Cref{cor:2} and conclude that $G_{W_0}\times G_{W^*_{n-1}}$ contains $C_k$ for any integer $k$ such that $5\leq k\leq 5^n$.
Since $G_{W_0}\times G_{W^*_{n-1}}$ is a spanning subgraph of $P(5^n)$, $C_k\subset P(5^n)$ for any integer $k$ such that $5\leq k\leq 5^n$.
Thus, if we can show that both $C_3$ and $C_4$ are subgraphs of $P(5^n)$, we can establish that $P(5^n)$ is pancyclic.

$P(5^n)$ is a strongly regular graph with parameters $(5^n, \frac{5^n-1}{2}, \frac{5^n-5}{4}, \frac{5^n-1}{4})$. 
By the assumption in the theorem, $n\geq 2$. 
In this case, we have 
\[\frac{5^n-5}{4}\geq \frac{5^2-5}{4}=5.\]
This means that any two vertices in $P(5^n)$ have at least 2 common neighbors.
Therefore, we can conclude that both $C_3$ and $C_4$ are indeed subgraphs of $P(5^n)$.
This implies $P(5^n)$ is pancyclic. 
From the above, we can deduce that $P(q)$ is pancyclic for any $q\neq 5$.
\end{proof}

\Cref{thm:main} determines the Paley index of $C_n$.
We will now prove \Cref{cor:index}.

\begin{proof}[Proof of \Cref{cor:index}]
  Let $\rho_{C_n}$ be the Paley graph of $C_n$.
  If $n>5$, then $P(\lceil n \rceil_{\FF})\neq P(5)$.
  Therefore, by \Cref{thm:main}, $P(\lceil n \rceil_{\FF})$ is pancyclic.
  By the definition of pancyclic and $n\leq \lceil n \rceil_{\FF}$, $C_n\subset P(\lceil n \rceil_{\FF})$.
  Hence, $\rho_{C_n}=\lceil n \rceil_{\FF}$.

  Since $P(5)=C_5$ and $C_5$ does not contain $C_3$ or $C_4$, we have $\rho_{C_5}=5=\lceil 5 \rceil_{\FF}$, and $\rho_{C_3}>5,\rho_{C_4}>5$.
  $P(9)$, which is the Paley graph with the second smallest number of vertices, is also pancyclic.
  Therefore, we have $\rho_{C_3}\leq 9,\rho_{C_4}\leq 9$.
  Thus, 
  \[
    \rho_{C_n}=
      \begin{cases*}
        \lceil n \rceil_{\FF} & $n\geq 5$\\
        9 & $n<5$
      \end{cases*},
  \]
  and the proof is complete.
\end{proof}

\begin{rem}
\Cref{cor:panc} provides one of the simplest methods, only counting the number of edges, to roughly verify the pancyclicity of any graph.
Similarly, \Cref{thm:cayPanc} is a newer method that is also simple. It only involves counting the number of common neighbors of any two vertices,  
thereby allowing one to roughly verify the pancyclicity of a specific Cayley graph.
It seems likely that there are other types of graphs whose pancyclic nature can be easily and roughly verified using these methods.
\end{rem}
\begin{rem}
We have determined the Paley-index of $C_n$, but a problem still persists.
The Paley-index of a graph $G$ is defined as the minimum order of $P(q)$ such that $G$ is a subgraph of $P(q)$.
Similarly, we can define the induced Paley index as the minimum order of $P(q)$ such that $G$ is an induced subgraph of $P(q)$.

\begin{df}[\cite{PaleyInd}]\label{def:PaleyInd2}
  We say $G$ has an induced Paley index $t$ if 
  \[
  t=\min\{q\in \NN\mid \mbox{$G$ is an induced subgraph of $P(q)$}\}. 
  \]
\end{df}

The induced Paley index of $C_n$ remains an open problem.
\end{rem}

\section*{Acknowledgments}

The author thanks to Professors Tsuyoshi~Miezaki, Akihiro~Munemasa and Tomoki~Yamashita for their helpful discussions and comments.

% \bibliographystyle{plain} 
% \bibliography{pancyclic.bib}

\begin{thebibliography}{1}

  \bibitem{BOLLOBAS198113}
  B.~Bollobás and A.~Thomason.
  \newblock Graphs which contain all small graphs.
  \newblock {\em European Journal of Combinatorics}, 2(1):13--15, 1981.
  
  \bibitem{BONDY197180}
  J.A Bondy.
  \newblock Pancyclic graphs {I}.
  \newblock {\em Journal of Combinatorial Theory, Series B}, 11(1):80--84, 1971.
  
  \bibitem{MR1829620}
  C.~Godsil and G.~Royle.
  \newblock {\em Algebraic graph theory}, volume 207 of {\em Graduate Texts in Mathematics}, pages 221--222.
  \newblock Springer-Verlag, New York, 2001.
  
  \bibitem{10.1307/mmj/1242071694}
  Tian~Khoon Kim and Cheryl~E. Praeger.
  \newblock {On generalised Paley graphs and their automorphism groups}.
  \newblock {\em Michigan Mathematical Journal}, 58(1):293 -- 308, 2009.
  
  \bibitem{TYamashita}
  R.~Matsubara, M.~Tsugaki, and T.~Yamashita.
  \newblock {\em A neighborhood and degree condition for
  pancyclicity and vertex pancyclicity}, volume 40 of {\em Australasian Journal of Combinatorics}, pages 15--25, 2008

  \bibitem{next}
  Y.~Nishimura 
  \newblock A new approach to pancyclicity of Paley graphs II.
  \newblock in preparation.

  \bibitem{PaleyInd}
  T.~Sakuma, T.~Miezaki, A.Munemasa and S.~Tsujie.
  \newblock Universal graph index and chromatic function.
  \newblock in preparation.

  \end{thebibliography}

\end{document}